\documentclass{llncs}
\usepackage{mathtools}
\usepackage{graphicx,amssymb,epsfig}
\usepackage[english]{babel}
\usepackage{caption} 
\usepackage{booktabs} 
\usepackage{varwidth}
\usepackage{epic,eepic}
\usepackage[utf8]{inputenc}
\def\I{\mathop{\mbox{\rm I}}}
\def\J{\mathop{\mbox{\rm J}}}
\def\Ext{\mathop{\mbox{\rm Ext}}}

\def\xor{\mathop{\mbox{\rm xor}}}
\def\neu#1{{\bf #1}}
\def\KK{\mathbb{K}}

\def\frak#1{\mathfrak{#1}}

\def\cal#1{\mathcal{#1}}

\def\dotcup{\mathop{\mathaccent\cdot\cup}} 
\newlength{\bigcupwidth}
\begin{document}
\title{On the Size of $\exists$-Generalized Concepts}
\author{Léonard Kwuida~\inst{1} \and Rostand S. Kuitché~\inst{2} \and Romuald E. A. Temgoua~\inst{3}}
\institute{Bern University of Applied Sciences,\\ Brückenstrasse 73, 3005 Bern, Suisse \\
\email{leonard.kwuida@bfh.ch}
\and
Université de Yaoundé I,
Département des Mathematiques, \\
 BP 812 Yaoundé, Cameroun
 \and 
Université de Yaoundé I,
École Normale Supérieure, \\
 BP 47 Yaoundé, Cameroun}
\maketitle

\begin{abstract}
Formal Concept Analysis (FCA) offers several tools for qualitative data analysis. One possibility is to group objects that share common attributes together and get a concept lattice that describes the data. Quite often the size of this concept lattice is very large. Many authors have investigated methods to reduce the size of this lattice. In \cite{KMBV14} the authors consider putting together some attributes to reduce the size of the attribute sets. But this reduction does not always carry over the set of concepts. They have provided some counter examples where the size of the concept lattice increases by one after putting two attributes together. Then they asked the following question: "How many new concepts can be generated by an $\exists$-generalization on just two attributes?" The present paper provides a family of contexts for which the size increases on more than one concept after putting solely two attributes together.
\end{abstract}
 
\section{Introduction}
An elementary information system can be represented by a set $G$ of objects or entities, a set $M$ of attributes or characteristics together with an incidence relation $\I$ that encodes whether an object $g\in G$ has an attribute $m\in M$. For such a system we write $(g,m)\in\I$ or $g\I m$ to mean that the object $g$ has the attribute $m$. The binary relation $\KK:=(G,M,\I)$ is called a  \neu{formal context}, and describes an elementary information system.  

To extract knowledge from such information systems, one possibility is to get clusters of objects and/or attributes by grouping together those sharing the same characteristics. These pairs, called \neu{concepts}, were formalized by Rudolf Wille~\cite{Wi82}. Traditional philosophers consider a concept as defined by two parts: an extent and an intent. The \neu{extent} contains all entities belonging to the concept and the \neu{intent} is the set of all attributes common to all entities in the concept.  
To formalize the notion of concept the following notations has been introduced for $A\subseteq G$ and $B\subseteq M$, known as \neu{derivation operators} in Formal Concept Analysis:
\begin{align*}
A^{\prime }=\{m\in M \mid (g,m)\in \I\text{ for all }g\in A\}, \\
B^{\prime }=\{g\in M \mid (g,m)\in \I\text{ for all }m\in B\}, \\ 
\end{align*}
where $\I$ denotes the corresponding \neu{incidence relation}.
$A^\prime$ contains all attributes shared by the objects in $A$ altogether. $B^\prime$ contains all objects satisfying all the attributes in $B$. Thus a concept is a pair $(A,B)$ with $A^\prime=B$ and $B^\prime =A$. The extents are then subsets $A$ of $G$ with $A^{\prime\prime}=A$, and intents subsets $B$ of $M$ with $B^{\prime\prime}=B$. For a single object or attribute $x$ we write $x^\prime$ for $\{x\}^\prime$. The map $X\mapsto X''$ is a \neu{closure operator} (on $\cal{P}(G)$ or $\cal{P}(M)$) and $X''$ is called the closure of $X$ in $\KK$. Subsets $X$ with $X''=X$ (i.e extents and intents) are closed subsets. 

We will denote the set of formal concepts of a context $\KK$ by $\frak{B}(\KK)$ and the set of its extents by $\Ext(\KK)$.
 A concept $c_2:=\left(A_{2},B_{2}\right)$ is said to be more general than a concept $c_1:=\left(A_{1},B_{1}\right)$ 
if $c_2$ contains all objects of $c_1$. In that case each attribute satisfied by all objects of $c_2$ is also satisfied by all objects of $c_1$. 
\begin{align*}
\left(A_{1},B_{1}\right)\leqslant\left(A_{2},B_{2}\right):\iff A_{1}\subseteq A_2,
\quad \mbox{(or equivalently   $B_{1}\supseteq B_{2}$)}.
\end{align*}
The  relation $\leqslant$ on concepts is an order relation called \neu{concept  hierarchy}. Each subset of $\frak{B}(\KK)$ has a supremum and an infimum with respect to $\leqslant$. Therefore $\left(\frak{B}(\KK),\leqslant\right)$ is a \neu{complete lattice} called the  \neu{concept lattice} of the context $\KK$.

\begin{table}[ht]
\begin{minipage}[b]{0.56\linewidth}
	\begin{center}
	\begin{tabular}{|l||l|l|l|l|} \hline
	$\KK$& $v$ & $u$ & $a$ & $b$ \\ \hline \hline
	$a$ &  & $\times $ & $\times $ &  \\ \hline
	$b$ & $\times $ &  &  & $\times $ \\ \hline
	$c$ & $\times $ & $\times $ &  & \\ \hline
	$g$ & $\times $ & $\times $ & $\times $  & $\times $ \\ \hline
	\end{tabular}%
\vspace{5mm}
	\caption{A formal context.}\label{tab:1}
	\end{center}
\end{minipage}\hfill
\begin{minipage}[b]{0.4\linewidth}
		\centering
	\includegraphics[scale=.5]{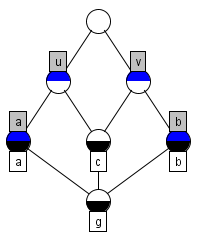}
		\captionof{figure}{A concept lattice} \label{fig:1}
	\end{minipage}
\end{table}
\noindent
Figure~\ref{fig:1} shows the concept lattice of the formal context in Table~\ref{tab:1}. 
Concepts are nodes. The extent of a node contains all objects below this node, and its intent all objects above it. The node at the center of this figure represents the concept $(\{c,g\},\{u,v\})$.

The size of concept lattices can be very large, even exponential with respect to the size of the context. For example the context $(E,E\neq)$, where $E$ is any set, has $2^{|E|}$ concepts. In fact $A^\prime=E\setminus A$ and $A^{\prime\prime}=A$ for any subset $A$ of $E$. Thus all pairs $(A,E\setminus A)$ are concepts of $(E,E\neq)$.

To control the size of concept lattices several methods have been suggested: decomposition~\cite{Wi85,Wi87,Wi89} , iceberg lattices~\cite{STBPL}, $\alpha$-Galois lattices~\cite{VS04}, fault tolerant patterns~\cite{BPRB}, similarity measures~\cite{AB11}, closure or kernel operators and/or approximation~\cite{Kw08},  generalized attributes~\cite{KMBV14}. In the present contribution we are following the direction in~\cite{KMBV14}, where some attributes are put together to defined a generalized attribute. 

When some attributes are put together, the main issue is to decide when an object has this new combined attribute. Different scenarios have been discussed in~\cite{KMBV14}:
\begin{itemize}
\item[$(\forall)$:] The object should satisfy each of the attributes that were combined.
\item[$(\alpha)$:] The object should satisfy at least a certain proportion of the attributes that were combined.
\item[$(\exists)$:] The object should satisfy at least one of the attributes that were combined. 
\end{itemize} 
By putting together some attributes we reduce the number of attributes and hope to also reduce the size of the concept lattice. This is true for $\forall$-generalizations, but not always the case for $\exists$-generalizations.  
In~\cite{KMBV14} some examples were presented where the size augment by one after a $\exists$-generalization. Their authors then asked whether the size can increase by more than one element after putting solely two attributes together. The present paper gives a positive answer to this question. In fact, we provide a family of contexts where the increase is exponential in the size of the attribute set.
  
Generalizing two attributes $m_1,m_2$ to get a new attribute $m_{12}$ can be done in two steps: (i): adding $m_{12}$ to the initial context and (ii) removing $m_1,m_2$ from the context. Therefore we start by discussing in Section~\ref{s:new_a} the effect of adding a new attribute in a context $\KK$. The main result here says that the maximal number of new concepts is $|\frak{B}(\KK)|$ \footnote{By $|X|$ we denote the number of elements in the set $X$.} and can be reached. This means that adding a new attribute to $\KK$ can double the size of $\frak{B}(\KK)$.
 In Section~\ref{s:increase} we present a family of contexts where the size increases by more than one after putting two attributes together, and by then answer the question raised in \cite{KMBV14}. Finally, we show in Section~\ref{s:max_increase} that the increase from Section~\ref{s:increase} is the maximum. The last section concludes the works and present further directions to be investigated. 
 
\section{Adding a new attribute into a context} \label{s:new_a}
When constructing concept lattices the incremental methods \cite{GMA95,VML02} consist in starting with one object/attribute and adding the rest of objects/attri\-butes one after another. In this section we review the effect of adding one attri\-bute. Let $\KK:=(G,M,\I)$ be a context, and $a\notin M$ an attribute that can be shared by some elements of $G$.  
We set $M_a:=M\dotcup\{a\}$ and $\KK_a:=(G,M_a,\I_a)$ where 
$$\I{}_a:=\I\cup\{(g,a)\mid g \mbox{ has the new attribute } a\}.$$ To distinguish between the derivation on sets of objects in $\KK$ and in $\KK_a$ we will use the name of the relation instead of $'$. That said we will write for $A\subseteq G$
\[A^{\I}=\{m\in M\mid g\I m \text{ for all }g\in A\}\] 
and 
\[A^{\I_a}=\{m\in M\cup\{a\}\mid g\I m \text{ for all }g\in A\}. \]
This distinction is not necessary on sets of attributes of $\KK$, unless we are looking for their closures.
 
If $a'=G$, then $\left|\frak{B}\left(\KK_a\right)\right|$ = $\left|\frak{B}(\KK)\right|$.  Each concept $(A,B)$ of $\KK$ has a corresponding concept $(A,B\dotcup\{a\})$ in $\KK_a$, and vice-versa. The above equality still holds even if $a'\neq G$, but $a'=B'$ for some $B\subseteq M$.

\begin{proposition} \label{p:prop1}
Let $\KK$ be a formal context and $\KK_a$ the formal context obtained by adding the attribute $a$ to $\KK$.
The map
\begin{equation*}
\begin{tabular}{llll}
$\phi _{a}:$ & $\frak{B}(\KK)$ & $\longrightarrow $ & $\frak{B}(\KK_a)$ \\
& $(A,B)$ & $\longmapsto $ & 
$\begin{cases}
(A,B\dotcup\{a\}) & \mbox{ if } A\subseteq a' \\
(A,B) & \mbox{else }%
\end{cases}%
$
\end{tabular}%
\end{equation*}%
is an injective map.
\end{proposition}
\begin{proof}
The map $\phi _{a}$ is well defined. In fact, for a concept $(A,B)\in \frak{B}(\KK)$ with $A\subseteq a'$,  we have $\left(B\dotcup\{a\}\right)' = B' \cap a' = A\cap a' = A$, and $A^{\I_a}=A^{\I}\cup\{a\}=B\dotcup\{a\}$. Thus $(A,B\dotcup\{a\})$ is a concept of $\KK_a$. 
For a concept $(A,B)\in \frak{B}(\KK)$ with $A\nsubseteq a'$,  we have $B'= A$, and $A^{\I_a}=A^{\I}=B$, since $a$ is not in $A^{\I_a}$. The injectivity of the map $\phi_a$ is trivial. If two concepts $(A_1,B_1)$ and $(A_2,B_2)$ of $\KK$ have the same image under $\phi_a$, then $A_1$ and $A_2$ are both included in $a'$ or both not included in $a'$, and are therefore equal. \qed
\end{proof}
After adding an attribute $a$ to a context $\KK$, we will identify $(A,B)\in\frak{B}(\KK)$ with $\phi_a(A,B)\in\frak{B}(\KK_a)$, and write $(A,B)\equiv \phi_a(A,B)$. From Proposition~\ref{p:prop1} we get $\left|\frak{B}(\KK)\right| \le \left|\frak{B}(\KK_a)\right|$. Moreover, the increase due to adding $a$, which is the difference $\left|\frak{B}(\KK_a)\right| - \left|\frak{B}(\KK)\right|$,  can be computed as the number of concepts of $\KK_a$ that cannot be identified  (via $\phi_a$) with any concept in $\frak{B}(\KK)$. 

We consider $(A,B)$ in $\frak{B}(\KK)$ with $A\nsubseteq a^\prime$. It holds
\[
\frak{B}(\KK_a)\ni (A,B)\equiv (A,B)\in\frak{B}(\KK), \ \text{since } A\nsubseteq a'.
\]
Moreover, $A\cap a'$ is an extent of $\KK_a$. If $A\cap a'$ is also an extent of $\KK$, then 
\[
\frak{B}(\KK_a)\ni\left(A\cap a',(A\cap a')^{\I_a}\right)\equiv \left(A\cap a',(A\cap a')^{\I}\right)\in\frak{B}(\KK) \quad \text{ since } A\cap a'\subseteq a'.
\]
Note that $(A\cap a^\prime)^{\I_a} =(A\cap a')^{\I}\dotcup \{a\}$ and $(A\cap a')^{\I} =(A\cap a^\prime)^{\I_a}\cap M$. 
Although $(A,B)$ and $(A\cap a^\prime, (A\cap a')^{\I}\dotcup\{a\})$ are two different concepts of $\KK_a$, they are equivalent to two concepts of $\KK$ when $A\cap a'$ is an extent of $\KK$. 
 A concept $(A,B)$ of $\KK$ induces two concepts of $\KK_a$ whenever $A\nsubseteq a^\prime$. In the definition of $\phi_a$ in Proposition~\ref{p:prop1} we went for $(A,B)$ instead of $(A\cap a', (A\cap a')^{\I}\dotcup\{a\})$. This choice is motivated by the injectivity of $\phi_a$ being straightforward.
If $A\nsubseteq a'$ and $A\cap a'$ is an extent of $\KK$ then the two concepts induced by $(A,B)$  in $\KK_a$ have their equivalent in $\frak{B}(\KK)$. 
Then adding $a$ to $\KK$ will increase the size of the concept lattice only if there is $A$ in $\Ext(\KK)$ such that $A\cap a'$ is not in $\Ext(\KK)$. 

Each extent of $\KK_a$ is an extent of $\KK$ or an intersection of an extent of $\KK$ with $a'$. The concepts of $\KK_a$ that cannot be identified (via $\phi_a$) to a concept of $\KK$ are 
\begin{align*}
\left\{\left(A\cap a^\prime, (A\cap a^\prime)^{\I}\dotcup\{a\}\right) \mid A\in\Ext(\KK) \text{ and } A\cap a^\prime \notin\Ext(\KK)\right\}.
\end{align*}
Note that it is possible to have two different extents $A_1, A_2\in\Ext(\KK)$ with $A_1\cap a' = A_2\cap a' \notin \Ext(\KK)$.  In this case we say that the extents $A_1$ and $A_2$ coincide on $a'$.  The increase is then less or equal to $|\frak{B}(\KK)|$.  We can now sum up the finding of the above discussion in the next proposition.
\begin{proposition} \label{p:prop2} Let $\KK_a$ be a context obtained by adding an attribute $a$ to a context $\KK$. We set
\begin{align}
\mathcal{H}(a):=
\left\{A\cap a^\prime \mid A\in\Ext(\KK) \text{ and } A\cap a^\prime \notin\Ext(\KK)\right\} \ \text{ and }\  h(a):=\left|\mathcal{H}(a)\right|.
\end{align}
\begin{enumerate}
\item The increase in the number of concepts due to adding the attribute $a$ to $\KK$ is
\[
 \left|\frak{B}\left(\KK_a\right)\right|-|\frak{B}(\KK)|= h(a) \le |\frak{B}(\KK)|
\]
\item The maximal increase is  $h(a)=|\frak{B}(\KK)|$ and is reached if each $A\in\Ext(\KK)$ satisfies $A\cap a^\prime\notin\Ext(\KK)$ and no pairs $A_1, A_2 \in\Ext(\KK)$ 
coincide on $a'$. 
\end{enumerate}
\end{proposition}
Before we continue with the discussion on the maximal increase, let us look at two examples, where an attribute $m$ has been added to the context of Table~\ref{tab:1}. 
In the first case (left of Fig.\ref{fig:2}) the concept lattice of Fig.~\ref{fig:1}  has been doubled and the maximal increase is reached. In the second case only the concepts in the interval $[\emptyset'';\{a,c\}'']$ of the concept lattice of Fig.~\ref{fig:1} has been doubled. Note that in both cases, $g\in\emptyset''\neq \emptyset$.

\begin{figure}[h] 
\begin{center}
	\begin{tabular}{|l||l|l|l|l|l|} \hline
		$\KK$& $v$ & $u$ & $a$ & $b$ &$m$\\ \hline \hline
		$a$ &  & $\times $ & $\times $ & & $\times $  \\ \hline
		$b$ & $\times $ &  &  & $\times $ & $\times $ \\ \hline
		$c$ & $\times $ & $\times $ &  & &$\times $ \\ \hline
		$g$ & $\times $ & $\times $ & $\times $  & $\times $ & \\ \hline
	\end{tabular}%
	\quad\hspace{2cm}
	\begin{tabular}{|l||l|l|l|l|l|} \hline
		$\KK$& $v$ & $u$ & $a$ & $b$ &$m$\\ \hline \hline
		$a$ &  & $\times $ & $\times $ & & $\times $  \\ \hline
		$b$ & $\times $ &  &  & $\times $ &  \\ \hline
		$c$ & $\times $ & $\times $ &  & &$\times $ \\ \hline
		$g$ & $\times $ & $\times $ & $\times $  & $\times $ & \\ \hline
	\end{tabular}%

\bigskip
	\includegraphics[scale=.5]{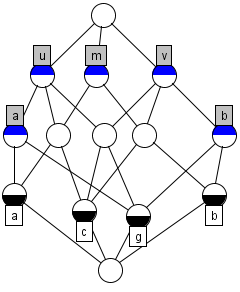}
	\quad\hspace{2cm}
	\includegraphics[scale=.5]{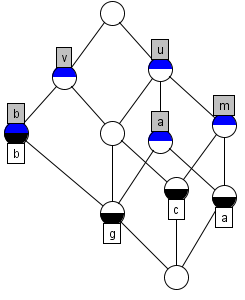}
\end{center}
	\caption{Two tables obtained by adding an attribute $m$ to the context in Table~\ref{tab:1}, and their corresponding concept lattices.} \label{fig:2}
\end{figure}

Based on the examples in Fig.~\ref{fig:2} and  Proposition~\ref{p:prop2}, we can now discuss the maximal increase.  Note that if $A\in\Ext(\KK)$ and $A\cap a' \notin\Ext(\KK)$, then $A\nsubseteq a'$. Moreover, if $A\nsubseteq a'$ for every extent $A$ of $\KK$, then in particular $\emptyset''\nsubseteq a'$. Thus there is $g\in \emptyset''$ such that $g\notin a'$. This element $g$ is in every extent of $\KK$, but is not in $a'$. Conversely, if an element $g$ is in $\emptyset ''\setminus  a'$, then $g$ is in every extent $A$ of $\KK$, and $g$ is not in $A\cap a'$. Thus $g\in \left(A\cap a'\right)^{\I\I}$ and $g\notin A\cap a'$, i.e. $A\cap a'$ is not closed in $\KK$. Thus $A\cap a'\notin\Ext(\KK)$ for each $A\in\Ext(\KK)$. 

\begin{proposition} \label{p:prop3} Let $\KK$ be a formal context and $a$ an attribute added to $\KK$. The following are equivalent:
	\begin{itemize}
\item[(i)] For every extent $A$ of $\KK$, $A\cap a'$ is not an extent of $\KK$.
\item[(ii)] $\emptyset^{\I\I}\setminus a'\ne\emptyset$.
	\end{itemize}
\end{proposition}
Both contexts of Fig.~\ref{fig:2} satisfy the above conditions (the added attribute $a$ is $m$). Each extent of $\KK$ generates two extents of $\KK_m$ and one of these cannot be identified (via $\phi_m$) with an extent of $\KK$. However, some of these new concepts can be equal. In fact if two extents coincide on $m'$, then they generate the same new concept. To avoid coincidences on $m'$, it is enough to have $m'=G\setminus \{g\}$.

\begin{corollary} \label{c:cor1}
	Let $\KK$ be a formal context such that $\emptyset ''\neq\emptyset $  and $\KK_a$ a context obtained by adding an attribute $a$ to $\KK$ such that  $a'=G\setminus\emptyset''$. Then we have $$\left|\frak{B}(\KK_a)\right| = 2\cdot \left|\frak{B}(\KK)\right|.$$
\end{corollary}


Using these results we can now present some huge increases after generalizing only two attributes.

\section{Number of concepts generated by an $\exists $-generalization} \label{s:increase}
Let $\KK:=(G,M,\I)$ be a formal context. A generalized attribute of $\KK$ is a subset $s\subseteq M$. We denote by $S$ the set of generalized attributes of $\KK$. Since the final goal is to reduce the size of the lattice, we will assume that $S$ forms a partition of $M$~\footnote{It is also possible to allow some attributes to appear in different groups. In this case the number of generalized attributes can be larger than in the initial context}. 
 Then at least the number of attributes is reduced. For a $\exists$-generalization we say that an object $g$ has the generalized attribute $s$ iff $g$ has at least one of the attributes in $s$; i.e. $s'=\bigcup\{m'\mid m\in s\}$. We get a relation $\J$ on $G\times S$ defined by: 
 \[g\J s\iff \exists m\in s \text{ such that } g\I m.\] 
 In this section we look at a very simple case, where two attributes $a,b\in M$ are generalized to get a new one, say $s$. This means that from a context $(G,M,\I)$, we remove the attributes $a$ and $b$ from $M$ and add an attribute $s\notin M$ to $M$, with $s'=a'\cup b'$. 
 In particular we show that the number of concepts of $(G,(M\setminus\{a,b\})\dotcup\{s\},\I_s)$ can be extremely larger than that of $(G,M,\I)$.
%
%
\begin{figure}[h!]
	\begin{center}
		\begin{tabular}{|l||l|l|l|l|} \hline
			$\KK^1_2$& $1$ & $2$ & $m_{1}$ & $m_{2}$ \\ \hline \hline
			$1$ &  & $\times $ & $\times $ &  \\ \hline
			$2$ & $\times $ &  &  & $\times $ \\ \hline
			$g_{1}$ & $\times $ & $\times $ &  & \\ \hline
		\end{tabular}%
		\qquad \qquad
		\begin{tabular}{|l||l|l|l|l|} \hline
			$\KK^1_{2\text{ge}}$& $1$ & $2$ & $m_{12}$  \\ \hline \hline
			$1$ &  & $\times $ & $\times $   \\ \hline
			$2$ & $\times $ &  &   $\times $ \\ \hline
			$g_{1}$ & $\times $ & $\times $   & \\ \hline
		\end{tabular}%

	\includegraphics[scale=.45]{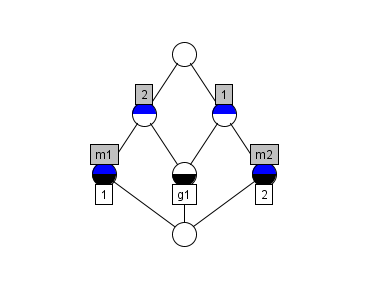} 
	\hspace{.1cm}
	\includegraphics[scale=.45]{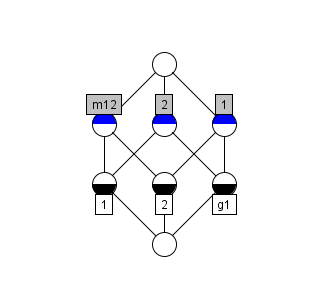} 
		\end{center} 
	\caption{$\frak{B}(\KK^1_2)$ (left) and $\frak{B}(\KK^1_{2\text{ge}})$ (right), as defined by~(\ref{e:k1g}) with $n=2$.} \label{fig:K12n=2}
\end{figure}

\noindent
By $S_{n}$ we denote a set with $n$ elements where $n\geq 2$, and write for simplicity
$S_{n}:=\{1,2,\cdots ,n\}$. We define a context $\KK^1_n$ by:
\begin{equation*}
\mathbb{K}^1_{n}:=(S_{n}\mathop{\mathaccent\cdot\cup}\{g_{1}\},S_{n}%
\mathop{\mathaccent\cdot\cup}\{m_{1},m_{2}\},\mathop{\mbox{\rm I}}) \quad\text{with}
\end{equation*}%

\begin{equation} \label{e:k1g}
g\mathop{\mbox{\rm I}}m:\iff
\begin{cases}
g,m\in S_{n} & \text{ and }g\neq m,\text{ or } \\
g=g_{1} & \text{ and }m\in S_{n},\text{ or } \\
g=1 & \text{ and }m=m_{1},\text{ or } \\
g\in S_{n}\setminus \{1\} & \text{ and }m=m_{2}.%
\end{cases}%
\end{equation}%
We generalize the attributes $m_{1}$ and $m_{2}$ to get $m_{12}$ and denote the resulting context by 
$\KK^1_{n\textup{ge}}:=(S_{n}\dotcup\{g_{1}\},S_{n} \dotcup\{m_{12}\},\I)$ with $m'_{12}=m'_1\cup m'_2$.
For the case $n=2$, the contexts and their concept lattices are displayed in Figure~\ref{fig:K12n=2}.

We want to compare the numbers of concepts of $\KK^1_{n\textup{ge}}$ and  that of $\KK^1_{n}$ and get the differences. The table below show some of these numbers:
\begin{table}
	\begin{center}
\begin{varwidth}[b]{0.6\linewidth}
	\begin{tabular}{c r r r r r r r r r} 
		\toprule
		$n$ & 2 & 3 & 4 & 5 & $\cdots$ & 10 & $\cdots$ & 20 & $\cdots$  \\ \hline
		$\left|\frak{B}(\KK^1_n)\right|$ & 7 & 13 & 25 & 49 & $\cdots$ & 1537 &$\cdots$ & 1572865 & $\cdots$   \\ \midrule
		$\left|\frak{B}(\KK^1_{n\textup{ge}})\right|$ & 8 & 16 & 32 & 64 &$\cdots$   & 2048 &$\cdots$ & 2097152 & $\cdots$  \\ \midrule
		$\left|\frak{B}(\KK^1_{n\textup{ge}})\right|-\left|\frak{B}(\KK^1_n)\right|$
		&1&3&7& 15 & $\cdots$  & 511 &$\cdots$ & 524287 & $\cdots$  \\ \bottomrule
	\end{tabular}
\end{varwidth}%
	\end{center}
\caption{Examples of increase after a $\exists$-generalization.}		
\end{table}

\paragraph{Notations:} We denote by $\I$ the restriction of the incidence relation of $\KK^1_n$ on any subcontext of $\KK^1_n$, and also by $\I$ the  incidence relation in the generalized context $\KK^1_{n\text{ge}}$. We set
\begin{align*}
\KK_{00}:=& (S_n\dotcup\{g_1\},S_n,\I),\\ 
\KK_{02}:=& (S_n\dotcup\{g_1\},S_n\dotcup\{m_2\},\I),\\ 
\KK_{01}:=& (S_n\dotcup\{g_1\},S_n\dotcup\{m_1\},\I), \\ 
\KK_{0s}:=& (S_n\dotcup\{g_1\},S_n\dotcup\{m_{12}\},\I) = \KK^1_{n\text{ge}},\\ 
\KK_{12}:=& (S_n\dotcup\{g_1\},S_n\dotcup\{m_1,m_2\},\I) = \KK^1_n.
\end{align*}

The context $\KK_{00}$ has $2^n$ concepts since $g_1$ is a reducible object in $\KK_{00}$ and the remaining context after removing $g_1$ is $(S_n,S_n,\neq)$. The context $\KK^1_n$ is obtained by adding successively $m_2$ to $\KK_{00}$ to get $\KK_{02}$, and then $m_1$ to $\KK_{02}$. The generalized context is obtained by adding $s=m_{12}$ to $\KK_{00}$. 

What happens when $m_2$ adds  to $\KK_{00}$? 
Every extent $A$ of $\KK_{00}$ is of the form $A=A_1\dotcup\{g_1\}$ with $A_1\subseteq S_n$ and satisfies ${A\cap m'_2\notin \Ext\left(\KK_{00}\right)}$. It therefore generates two concepts in $\KK_{02}$. The extents $A$ with $A_1\subseteq m'_2=\{2,\cdots n\}$ do not coincide on $m'_2$, and therefore generate $2^{n-1}$ concepts in $\KK_{02}$ that cannot be identified (via $\phi_{m_2}$) to any concept of $\KK_{00}$. If $A$ is an extent of  $\KK_{00}$ containing $1$ then $A\setminus\{1\}$ is also an extent of $\KK_{00}$, and both extents coincide on $m'_2$. Thus by Proposition~\ref{p:prop2} we get
  \begin{align*}
  \left|\frak{B}\left(\KK_{02}\right) \right|=2^n+2^{n-1}.
  \end{align*}
Now adding $m_1$ to $\KK_{02}$ generates at most two concepts, since $m'_1=\{1\}$ and 
\begin{align*}
\cal{H}(m_1)\subseteq 
\{A\cap m'_1\mid A\in\Ext\left(\KK_{02}\right)\} = \{\emptyset,m'_1\}
\end{align*}
The extents generated by $\{1\}$ in $\KK_{00}$ and in $\KK_{02}$ are equal to  $\{1,g_1\}$. 
Thus $\{1,g_1\}\cap m'_1 =m'_1 \notin\Ext\left(\KK_{02} \right)$. But $\emptyset\in \Ext\left( \KK_{02} \right)$ and $\emptyset\cap m'_1=\emptyset$. Therefore
  \begin{align*}
 \cal{H}(m_1)=\{m'_1\}\quad \text{ and }\quad
\left|\frak{B}(\KK_{12})\right|=2^n+2^{n-1}+1.
\end{align*}

\begin{figure}[h]
	\begin{center}
		\includegraphics[scale=0.7]{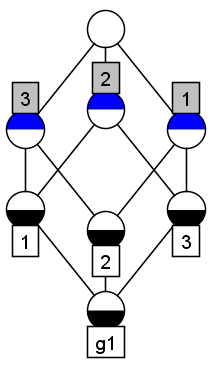} \hspace{5mm}
		\includegraphics[scale=0.7]{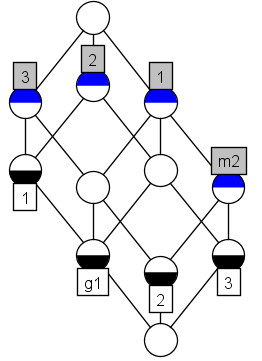}
	\end{center}

	\begin{center}
		\includegraphics[scale=1.1]{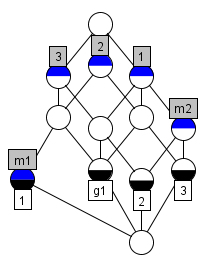} \hspace{5mm}
		\includegraphics[scale=1.1]{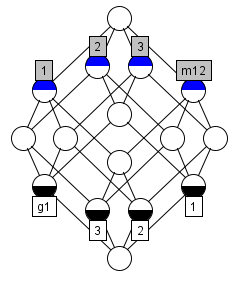}
	\end{center}
	\caption{$n=3$. \quad $\frak{B}\left(\KK_{00}\right)$ (upper left); \quad $\frak{B}\left(\KK_{02}\right)$ (upper right); \quad $\frak{B}\left(\KK_{12}\right)$ (down left)\quad and $\frak{B}\left(\KK_{0s}\right)$ (down right).
	}\label{fig:BK0012}
\end{figure}

 The context $(S_n\dotcup\{g_1\},S_n\dotcup\{m_{12}\},\I)$ is isomorphic to $(S_{n+1},S_{n+1},\neq)$. The object $g_1$ is identified with $n+1$ and the attribute $m_{12}$ with $n+1$. Thus generalizing $m_1$ and $m_2$ to $m_{12}$ leads to a lattice with $2^{n+1}$ concepts. The increase is then
 \begin{align*}
2^{n+1}-\left(2^n+2^{n-1}+1\right) = 2^{n-1}-1
 \end{align*}
 which is exponential in the number of attributes of the initial context. 
We can summarize the discussion above in the following proposition:
\begin{proposition} \label{p:prop3}
	Let $n\geq 2$ and $(S_n\dotcup\{g_1\},S_n\dotcup\{m_1,m_2\},\I)$ defined by (\ref{e:k1g}).  Generalizing the attributes $m_1$ and $m_2$ increases the size of the concept lattice by $2^{n-1}-1$. 
\end{proposition}

In the above construction of $\KK^1_{n\textup{ge}}$ the idea is to construct a context $(E,E\neq)$ with $|E|=n+1$ from the initial context. via a $\exists$-generalization. The objects in $S_n$ are split between $m_1$ and $m_2$ with no overlap.
We can choose a split that assigns $k$ objects of $S_n$ to $m_1$ and the other $n-k$ to $m_2$. Let \[\KK^k_{n}:=(S_{n}\cup
\{g_{1}\},S_{n}\cup \{m_{1},m_{2}\},\I) \]
 be such a context, where $\I$ is defined by
\begin{equation} \label{e:k2g} 
g\I m :\iff
\begin{cases}
g,m\in S_{n} & \mbox{ and }g\neq m, \mbox{ or }\\
g=g_{1} &\mbox{ and } m\in S_{n} \mbox{ or } \\
g\in \{1,2,...,k\} &\mbox{ and } m=m_{1} \mbox{ or } \\
g\in S_{n}\setminus \{1,2,...,k\} &\mbox{ and }m=m_{2}
\end{cases}
\end{equation}
Then the existential generalization of the attributes $m_{1}$ and $m_{2}$ to $m_{12}$ leads to the generalized context $\KK^k_{n\text{ge}}:=(S_{n}\cup \{g_{1}\},S_{n}\cup \{m_{12}\},\I)\cong (S_{n+1},S_{n+1},\neq)$.
To get the cardinality of $\frak{B}(\KK^k_n)$, we observe that
\begin{itemize}
	\item[(i)]   $\KK_{00}:=(S_{n}\cup \{g_{1}\},S_{n},\I)$ has $2^n$ concepts. The extents of $\KK_{00}$ are of the form $A\dotcup\{g_1\}$, $A\subseteq S_n$. 
	\item[(ii)]   $\KK_{02}:= (S_{n}\cup \{g_{1}\},S_{n}\cup \{m_{2}\},\I)$ has $2^n+2^{n-k}$ concepts. They are of the form $(A\dotcup\{g_1\},S_n\setminus A)$ with $A\subseteq S_n$ or the form $(A,(S_n\setminus A) \dotcup\{m_2\})$ with $A\subseteq m'_2$.  
	\item[(iii)] $\KK_{01}:=(S_{n}\cup \{g_{1}\},S_{n}\cup \{m_{1}\},\I)$ has $2^n+2^k$ concepts, which are of the form $(A\dotcup\{g_1\},S_n\setminus A)$ with $A\subseteq S_n$ or the form  $(A,(S_n\setminus A) \dotcup\{m_1\})$ with $A\subseteq m'_1$.
\end{itemize}

$\KK_{12}:=(S_{n}\cup \{g_{1}\},S_{n}\cup \{m_1,m_2\},\I)$ is obtained from $\KK_{02}$ by adding $m_1$. Therefore we need to compute $\cal{H}(m_1)$ with respect to $\KK_{02}$. Let  $A\in\Ext(\KK_{02})$. We distinguish two cases:

 \begin{itemize}
 	\item[(i)] If $g_1\notin A$, then $A\subseteq m'_2$, and $A\cap m'_1=\emptyset$ is an extent of $\KK_{02}$. No new concept is generated. 
 	\item[(ii)] If $g_1\in A$, then the extent $A$ is of the form $A=A_1\dotcup\{g_1\}$ with $A_1\subseteq S_n$. 
Since $m'_1\cap m'_2=\emptyset$, we get 
\begin{align*}
A_1\cap m'_1 \notin \Ext\left(\KK_{02}\right) &\iff A_1\cap m'_1 \nsubseteq m'_2 
\end{align*}
Thus the number of additional concept generated is 
\[\left|\left\{A\cap m'_1 \mid A\in\Ext \left(\KK_{02}\right) \text{ and } A\cap m'_1 \nsubseteq m'_2
\right\} \right|
\]
 Among the extents of $\KK_{02}$ with $A\cap m'_1 \nsubseteq m'_2$, there are $2^k-1$ that do not coincide on $m'_1$, for example those with $\emptyset\neq A_1\subseteq m'_1$. This means that adding $m_1$ to $\KK_{02}$ will generate $2^k-1$ new concepts that cannot be identified with concepts in $\KK_{02}$. Therefore $\KK^k_{n}$ has $2^{n}+2^{n-k}+2^{k}-1$ concepts. 
 \end{itemize}
\begin{proposition} \label{p:prop5}
	Let $n\geq 2$,  $1\le k <n$ and $\KK^k_n$ defined by (\ref{e:k2g}).
	\begin{itemize}
		\item[a)] The context $\KK^k_{n}$ has $2^{n}+2^{n-k}+2^{k}-1$ concepts.
		\item[b)] Generalizing $m_1$ and $m_2$ increases the number of concepts by $$2^{n}-2^{k}-2^{n-k}+1.$$
	\end{itemize}
\end{proposition}
\noindent
A natural question here is: which $\KK^k_n$ has a maximal increase? 
The increase by an $\exists$-generalization that puts $m_1$ and $m_2$ together in $\KK^k_n$ is $$f_n(k):=2^n-2^k-2^{n-k}+1.$$ This function is convex and its slope vanishes at $k=\frac{n}{2}$. 
	\begin{align*}
\frac{d}{dk}f_n(k) &= -\ln(2) 2^k+\ln(2)2^{n-k} = 0 \iff n=2k.\\
\frac{d^2}{dx^2}f_n(k) &= -\ln^2(2)2^k -\ln^2(2)2^{n-k} <0.
	\end{align*}
\noindent
 The maximum is reached when the objects are evenly split; i.e $k=\frac{n}{2}$ for $n$ even, or $k\in\left\{\lfloor \frac{n}{2}\rfloor, \lfloor \frac{n}{2}\rfloor +1\right\}$ for $n$ odd. That is the case for the context $$\KK^{\lfloor \frac{n}{2}\rfloor}_n:=(S_n\dotcup\{g_1\},S_n\dotcup\{m_1,m_2\},\I)$$ with \begin{align*}
g\I m \iff \begin{cases} 
g,m\in S_n &\text{ and } g\neq m, \text{ or }\\ 
g=g_1 &\text{ and } m\in S_n, \text{ or } \\
g\in \{1,\cdots,\lfloor \frac{n}{2} \rfloor\} & \text{ and } m= m_1, \text{ or } \\ 
g\in S_n\setminus \{1,\cdots,\lfloor\frac{n}{2}\rfloor\} &\text{ and } m= m_2.
\end{cases}
\end{align*}
If $n=2q$, then the increase is  $f_{2q}(q) = 2^{2q}-2\cdot 2^q+1 = \left(2^q-1\right)^2$. If $n=2q+1$, then the increase is  $f_{2q+1}(q) = 2^{2q+1}-2^q-2^{q+1}+1 = \left(2^q-1\right)\left(2^{q+1}-1\right)$.

We could allow overlap in constructing $\KK^k_n$ by using any covering of $S_n$ with two proper subsets $m'_1$ and $m'_2$; this means
$$
m_1'\cup m_2' = S_n \text{ with } \emptyset \subsetneq m'_1, m'_2\subsetneq S_n.
$$
An $\exists$-generalization that puts the attributes $m_1$ and $m_2$ together to get $m_{12}$, will also have $2^{n+1}$ concepts. However the concept lattice of $\KK_n$ will have more concepts when $m'_1\cap m'_2\neq\emptyset$ compared to when $m'_1\cap m'_2=\emptyset$. 
The minimal increase in that case is achieved with $\left| m'_1 \right| = n-2 =\left|m'_2 \right|$  and $\left| m'_1\cap m'_2 \right|=n-3$. 

Let $\KK_{12}:=(S_{n}\cup \{g_1\},S_{n}\cup \{m_{1},m_{2}\},\I)$ with $m'_1\cap m'_2\neq \emptyset$. If $m'_1\subseteq m'_2$ or $m'_2\subseteq m'_1$ then putting 
$m_{1}$ and $m_{2}$ together will not increase the size of the concept lattice.
 Therefore, we assume that  $m'_1\parallel m'_2$. 

\begin{proposition}\label{p:prop6}
	Let $n>2$ and $\KK_{s}$ the generalized context obtained from $\KK_{12}$ by putting $m_1$ and $m_2$ together. Then:
	\begin{enumerate}
		\item  The size of the concept lattice of the context $\KK_n$ is 
		 $$2^{n}+2^{|m'_{2}|}+2^{|m'_{1}|}-2^{|m'_2\cap m'_1|}.$$
		\item After the generalization, the size of the initial lattice increases by 
		\[2^{n}-2^{|m'_1|}-2^{|m'_2|}+2^{|m'_1\cap m'_2|}.\]
	\end{enumerate}
\end{proposition}
Before we start with the proof we look at a concrete case with $n=3$. 
Its context is isomorphic to 
\begin{center}
	\begin{tabular}{|l|l|l|l|l|l|} \hline
		 & $1$ & $2$ & $3$ & $m_{1}$ & $m_{2}$ \\ \hline
		$1$ &  & $\times $ & $\times $ & $\times $ &  \\ \hline
		$2$ & $\times $ &  & $\times $ & $\times $ & $\times $ \\ \hline
		$3$ & $\times $ & $\times $ &  &  & $\times $ \\ \hline
		$g_{1}$ & $\times $ & $\times $ & $\times $ &  & \\ \hline
	\end{tabular}
\end{center}
and has $14$ concepts. Note that $14=2^{3}+2^{2}+2^{2} -2^{1}$.

\begin{proof}
	$\KK_{12}:=(S_{n}\dotcup \{g_1\},S_{n}\dotcup \{m_{1},m_{2}\},\I)$ has $2^{n}+2^{|m'_1|} + \cal{H}(m_2)$ concepts, where $\cal{H}(m_2)$ is to be determined  with respect to $\KK_{01}:=(S_{n}\cup \{g_{1}\},S_{n}\cup \{m_1\},\I)$.
	The concepts of $\KK_{01}$ are of the form $\left(A_1\dotcup\{g_1\},S_n\setminus A_1\right)$ with $A_1\subseteq S_n$ or of the form $\left(A_1,(S_n\setminus A_1)\cup\{m_1\} \right)$ with $A_1\subseteq m'_1$. Let $A\in\Ext\left(\KK_{01}\right)$.
\begin{itemize}
	\item If $g_1\notin A$ then $A\subseteq m'_1$ and $A\cap m'_2$ is a subset of $m'_1$, and by then an extent of $\KK_{01}$. No new concept is generated. 
	\item If $g_1\in A$ then $A=A_1\dotcup\{g_1\}$ with $A_1\subseteq S_n$. Then 
	\[A\cap m'_2\notin\Ext\left(\KK_{01}\right) \iff A\cap m'_2\subseteq m'_2\setminus m'_1.\]
\end{itemize}
Thus, adding $m_2$ to $\KK_{01}$ will generate $2^{|m'_2|}-2^{|m'_1\cap m'_2|}$ new concepts that cannot be identified (via $\phi_{m_2}$) with concepts in $\KK_{01}$. Then $\KK_{12}$ has $$2^{n}+2^{|m'_1|}+2^{|m'_2|}-2^{|m'_1\cap m'_2|}$$ concepts.
The increase of the size of the lattice is then
\begin{align*}
\left|\frak{B}(\KK_{12}) \right|- \left|\frak{B}(\KK_{01}) \right| &= 2^{n+1}-\left(2^{n}+2^{|m'_1|}+2^{|m'_2|}-2^{|m'_1\cap m'_2|} \right) \\
 &= 2^{n}-2^{|m'_1|}-2^{|m'_2|}+2^{|m'_1\cap m'_2|}
\end{align*} \qed
\end{proof}

\begin{remark}  
Note that $n=|m'_1\cup m'_2|$ and the increase is 
\[ 2^{|m'_1\cup m'_2|}-2^{|m'_1|}-2^{|m'_2|}+2^{|m'_1\cap m'_2|}
\]
which is a general formula that holds, even if $m'_1\cup m'_2\neq S_n$.
The starting context is $\KK_{00}:=(S_n\dotcup\{g_1\},S_n,\I)$ and has $2^n$ extents. After adding an attribute $m_1$ to $\KK_{00}$ we increase the number of extents by $2^{|m'_1|}$. After adding $m_2$ to $\KK_{00}$ we increase the number of extents by $2^{|m'_2|}$. After adding an attribute $s$ with $s'=m'_1\cup m'_2$ 
to $\KK_{00}$ we increase the extents by $2^{|m'_1\cup m'_2|}$.
If we add an attribute $t$ with $t'=m'_1\cap m'_2$ to $\KK_{00}$ we will increase the extents by $2^{|m'_1\cap m'_2|}$. But these extents ''appear'' already when $m_1$ or $m_2$ is added to $\KK_{00}$, and are therefore counted twice when both $m_1$ and $m_2$ are added to $\KK_{00}$. 
\end{remark}

\begin{remark}
The counting with $\KK_{12}$ has been made easy by the fact that each ''subset'' of $S_n$ identifies an extent of $\KK_{00}$. If  $m'_1\cap m'_2$ is not empty, then $\KK_{12}$ has more concepts while the number of generalized concept remains the same. Then the condition $m'_1\cap m'_2=\emptyset$ is necessary (but not sufficient) to get the maximal increase. For the contexts $\KK_{12}$, putting $m_1$ and $m_2$ together can increase the the size of the concept lattice by up to $\left(2^{\lfloor \frac{n}{2}\rfloor }-1\right)\left(2^{\lceil\frac{n}{2}\rceil}-1\right)$ concepts. Is this the maximal increase for contexts of similar size?
\end{remark}

\begin{remark}
	Note that all contexts $\KK_{12}$ we have constructed are reduced. Requiring the contexts to be reduced is a fair assumption. If not then we should first remove reducible attributes before processing with a generalization. This removal does not affect the size of the concept lattice. However putting together two reducible attributes will for sure not decrease the size, but probably increases it.
\end{remark}

\begin{remark}
		$B_4$ is the smallest lattice for which there are two attributes whose $\exists$-generalization increases the size of the concept lattice. All lattices presented in this section contain a labelled copy of $B_{4}$ (as subposet!). Is there any characterization of contexts for which generalizing increases the size, for example in terms of forbidden subcontexts or subposets? 
\end{remark}
\begin{table}[ht]
	\begin{center}

	\end{center}
\begin{minipage}[b]{0.5\linewidth}
	\includegraphics[scale=.7]{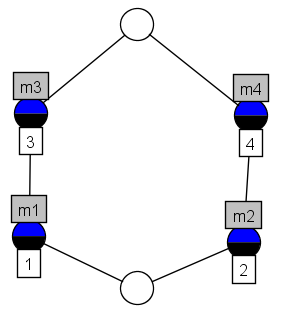}
	\captionof{figure}{A concept lattice for $B_4$.} \label{fig:B4}
\end{minipage}\hspace{3mm}
\begin{minipage}[b]{0.5\linewidth}
		\includegraphics[scale=.7]{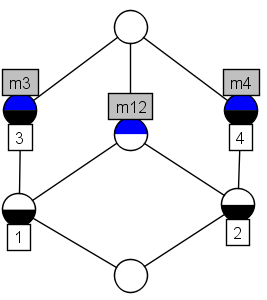}
		\captionof{figure}{Generalizing $m_1$ and $m_2$.} \label{fig:1}
	\end{minipage}
\end{table}


In this section we have found out that the size of the generalized concept lattice can be exponentially larger than that of the initial concept lattice
after an existential generalization. In the next section we will discuss the maximum of the increase after a $\exists$-generalization.

\section{Maximum increase after an existential generalization} \label{s:max_increase}
In this section we investigate the maximal increase in the general case. This means that from a context $\KK:=(G,M,\I)$ that is attribute reduced, two attributes $a,b$ are removed and replaced with an attribute $s$ defined by $s'=a'\cup b'$. We set $M_0=M\setminus\{a,b\}$ and adopt the following notation:
\begin{align*}
\KK_{00}:=& (G,M_0,\I), & \text{(removing $a,b$ from $\KK$)}\\ 
\KK_{01}:=& (G,M_0\dotcup\{a\},\I), & \text{(adding $a$ to $\KK_{00}$)} \\ 
\KK_{02}:=& (G,M_0\dotcup\{b\},\I), & \text{(adding $b$ to $\KK_{00}$)} \\ 
\KK_{0s}:=& (G,M_0\dotcup\{s\},\I), & \text{(generalized context)} \\ 
\KK_{12}:=& (G,M_0\dotcup\{a,b\},\I)=\KK. & \hfill \text{(initial context)}
\end{align*}
In general we get context $\KK$ by adding $a$ to $\KK_{00}$ and get $\KK_{01}$, and add $b$ to $\KK_{01}$. Recall that if an attribute $m$ is added to any context $\KK$, then the number of concepts increase by
 \[ h(m)=|\{A\cap m'\mid A\in\Ext(\KK) \text{ and } A\cap m'\notin \Ext(\KK)\}|
 \]
 We denote by $a\cap b$ the attribute defined by $(a\cap b)':=a'\cap b'$, and by $a\cup b$ the attribute defined by $(a\cup b)':=a'\cup b'=s'$.  We start from $\KK_{00}$. 
 Adding the attribute $a$ to $\KK_{00}$ increases its number of concepts by 
 \[h(a)=|\{A\cap a'\mid A\in \Ext(\KK_{00}) \text{ and } A\cap a'\notin \Ext(\KK_{00})  \}| \le 2^{|a'|}.
 \]
 Adding the attribute $b$ to $\KK_{00}$ increases its number of concepts by 
 \[h(b)=|\{A\cap b'\mid A\in \Ext(\KK_{00}) \text{ and } A\cap b'\notin \Ext(\KK_{00})  \}| \le 2^{|b'|}.
 \]
 Adding the attribute $a\cap b$ to $\KK_{00}$ increases its number of concepts by 
 \[h(a\cap b)=|\{A\cap a'\cap b' \mid A\in \Ext(\KK_{00}) \text{ and } A\cap a'\cap b'\notin \Ext(\KK_{00})  \}| \le 2^{|a'\cap b'|}.
 \]
 But these concepts appear in $\cal{H}(a)$ and $\cal{H}(b)$ and will be counted twice.  
 If $a'\cap b'$ is empty then $h(a\cap b)\le 1$. 

\bigskip 
\noindent
 Adding the attribute $a\cup b$ to $\KK_{00}$ increases its number of concepts by 
 \begin{align*}
 h(a\cup b) &=|\{A\cap (a'\cup b')\mid A\in \Ext(\KK_{00}) \text{ and } A\cap (a'\cup b')\notin \Ext(\KK_{00})\}|\\
 &\le 2^{|a'\cup b'|} \le 2^{|a'|+|b'|}.
 \end{align*}
 If $h(a\cup b)=2^{|a'|+|b'|}$ then $a'\cap b'=\emptyset$ and each subset of $a'\cup b'$ is not an extent of $\KK_{00}$, but is the restriction of an extent of $\KK_{00}$ on $a'\cup b'$. In this case $h(a)=2^{|a'|}$, $h(b)=2^{|b'|}$ and $h(a\cap b)=1$. 

 \bigskip
 \noindent
We denote by $h(a,b)$ the increase when two attributes $a$ and $b$ are both added to $\KK_{00}$. Then we have 
\begin{align*}
\left|\frak{B}(\KK_{12})\right| &= \left|\frak{B}(\KK_{01})\right|+ h(b)-h(a\cap b)\\ 
& =\left|\frak{B}(\KK_{00})\right| + h(a) + h(b)-h(a\cap b).
\end{align*}
and 
\[ h(a,b)=h(a)+h(b)-h(a\cap b).\]
The increase is then 
\begin{align*}
\left|\frak{B}(\KK_{0s})\right| -\left|\frak{B}(\KK)\right| &= h(a\cup b)-h(a,b)	\\
&= h(a\cup b) -h(a) - h(b)+h(a\cap b)
\end{align*}
If $h(a\cup b)=2^{|a'|+|b'|}$ then the increase is 
\[
h(a\cup b) -h(a) - h(b)+h(a\cap b) = 2^{|a'|+|b'|} -2^{|a'|} - 2^{|b'|} + 1.
\] 
Now we are going to show that this increase is the least upper bound. 
Since we are interested in the maximal increase, we assume that $a'\cap b'=\emptyset$. In fact $\KK$ has more concepts when $a'\cap b'\neq \emptyset$, than when $a'\cap b'=\emptyset$; But the number of concepts of $\KK_{0s}$ will remain the same in both cases. The increase $\left|\frak{B}(\KK_{0s})\right| -\left|\frak{B}(\KK)\right|$ is then larger if $\left|\frak{B}(\KK)\right|$ is smaller. 

\bigskip
\noindent
There exists $d_1\leq 2^{|a'|}$ such that $h(a)= 2^{|a'|}-d_1$. 
In fact $$d_1=|\{A\subseteq a'\mid A\in\Ext(\KK_{00}) \}|.$$ 
Similarly, there exists $d_2\leq 2^{|b'|}$ such that $h(b)= 2^{|b'|}-d_2$. 
As above we have  $$d_2=|\{A\subseteq b'\mid A\in\Ext(\KK_{00}) \}|.$$ 
Similarly, there exists $d_0\leq 2^{|a'\cup b'|}$ such that $|\cal{H}(s)|= 2^{|a'\cup b'|}-d_0=2^{|a'|+|b'|}-d_0$, 
with $$d_0=|\{A\subseteq a'\cup b'\mid A\in\Ext(\KK_{00}) \}|.$$ 
Since we assume $a'\cap b'=\emptyset$, the following holds for any extent $A\ne \emptyset$ of $\KK_{00}$: 
\begin{align*}
A\subseteq a'\cup b'\iff & A\subseteq a' \xor  A\subseteq b' \xor A\subseteq a'\cup b',\  A\nsubseteq a'\text{ and } A\nsubseteq b'.	
\end{align*}
where $\xor$ denotes the exclusive or. Therefore $d_1+d_2\le d_0 $ and the increase is then
\begin{align*}
\left|\frak{B}(\KK_{0s})\right| -\left|\frak{B}(\KK)\right| &= h(a\cup b)-h(a,b)	\\
&= h(a\cup b) -h(a) - h(b)+h(a\cap b)\\
&=\left(2^{|a'|+|b'|}-d_0\right) - \left(2^{|a'|}-d_1\right)-\left(2^{|b'|}-d_1\right)+h(a\cap b) \\
&=2^{|a'|+|b'|}-2^{|a'|}-2^{|b'|}+ h(a\cap b) + \underbrace{d_1+d_2-d_0}_{\le 0} \\
&\le 2^{|a'|+|b'|}-2^{|a'|}-2^{|b'|}+ h(a\cap b)\\
&\le 2^{|a'|+|b'|}-2^{|a'|}-2^{|b'|}+ 1,
\end{align*}
since $h(a\cap b)\in\{0,1\}$ when $a'\cap b'=\emptyset$.

\begin{theorem}
	Let $(G,M,\I)$ be an attribute reduced context 
	and $a, b$ two attributes such that their existential generalization $s=a\cup b$ increases the size of the concept lattice. 
	Then
	\begin{itemize}
\item[(i)] $|\frak{B}(G,M,\I)| = |\frak{B}(G,M\setminus\{a,b\},\I)|+h(a,b)$, with 
\[ h(a,b)=h(a)+h(b)-h(a\cap b).\]
\item[(ii)] The increase after generalizing is 
\begin{align*}
h(a\cup b)-h(a)-h(b)+h(a\cap b)
 \le 2^{|a'|+|b'|}-2^{|a'|}-2^{|b'|}+1.
\end{align*}
	\end{itemize}
\end{theorem}

\begin{remark}
	If generalizing $a$ and $b$ does not increase the size of the lattice, then the difference 
	\[h(a\cup b)-h(a)-h(b)+h(a\cap b)\]
	is at most zero, and will describe the reduction in the number of concepts. 
\end{remark}

\section{Conclusion} \label{s:conclusion}
In this work, we have shown  a family of concepts lattices in which an
existential generalization on a specific pair of attributes increases the
size of the lattice. We have also found the maximal increase when two attributes are put together. Characterizing contexts where such a generalization increases the size of the lattice is still to be done. Another direction of interest is to look at similarity measures that discriminate attributes if putting these together increases the  number of concepts. 

 
 \end{document}